\documentclass[twoside,reqno,10pt]{amsart}

\usepackage{amsmath, amssymb} 
\usepackage{graphicx, color}
\usepackage{url}

\newcommand{\llim}[3]{\ensuremath{ \lim\limits_{ #1 \rightarrow #2} #3 }}
\newcommand{\fclass}[2]{\ensuremath{  \mathbb{#1}^{\, #2} }}

\newtheorem{proposition}{Proposition}

\newtheorem{definition}{Definition}
\newtheorem{theorem}{Theorem}
\newtheorem{corollary}{Corollary}
\newtheorem{remark}{Remark}

\usepackage{listings}

\definecolor{dkgreen}{rgb}{0,0.6,0}
\definecolor{gray}{rgb}{0.5,0.5,0.5}
\definecolor{mauve}{rgb}{0.20 , 0.40, 1.0}

\lstdefinelanguage{Maxima}
{morekeywords={allbut, and, block, break, buildq, do, else, elseif, error, errcatch, for,   get, go, if, in, is, lambda, local, 
		matchdeclare, new, not, or, step, then, thru, tellsimp, tellsimpafter, unless, while, return},
	sensitive=false,
	comment=[s]{/*}{*/},
	morestring=[b]"
}

\usepackage{listings}
\usepackage{float}

\floatstyle{plaintop}
\newfloat{listing}{thp}{lop}
\floatname{listing}{Listing}

\begin{document}
	\title[Regularized integral representations of the reciprocal $\Gamma$ function] {Regularized integral representations of the reciprocal $\Gamma$ function}
	\author {Dimiter Prodanov }
	\address{Correspondence: Environment, Health and Safety, IMEC vzw, Kapeldreef 75, 3001 Leuven, Belgium; e-mail: Dimiter.Prodanov@imec.be, dimiterpp@gmail.com}
\begin{abstract}
	 This paper establishes a real integral representation of the reciprocal $\Gamma$ function in terms of a regularized hypersingular integral.
	 The equivalence with the usual complex representation is demonstrated.
	 A regularized complex representation along the usual Hankel path is derived.
\end{abstract}

\maketitle

\section{Introduction}\label{se:intro}

Applications of the Gamma function are ubiquitous in fractional calculus and the special function theory.
It has numerous interesting properties summarized in \cite{Borwein2018}.
The history of the Gamma function is surveyed in \cite{Davis1959}.
In a previous note I have investigated an approach to regularize derivatives at points, where the ordinary limit diverges \cite{Prodanov2016a}. 
This paper exploits the same approach for the purposes of numerical computation of singular integrals, such as the Euler $\Gamma$ integrals for negative arguments.
The paper also builds on an observations in  \cite{Mainardi2010a}.

The present paper proves a real singular integral representation of the reciprocal $\Gamma$ function. 
The algorithm is implemented in the Computer Algebra System Maxima for reference and demonstration purposes.

As a second contribution, the paper provides an integral representation of the Gamma function for negative numbers.
Finally, the paper demonstrates an equivalent regularized complex representation based on the regularization of the Heine integral.

\section{Preliminaries and notation}\label{sec:notation}

The reciprocal Gamma function is an entire function.
Starting from the Euler's infinite product definition, the reciprocal Gamma function can be defined by the infinite product:
\[
\frac{1}{\Gamma(z)}:= \llim{n}{\infty}{} \frac{z \left( z+1 \right) \ldots  \left( z+n \right)  }{ n^z \, n!}
\]
Proceeding from the Euler's reflection formula for negative arguments the reciprocal Gamma function is simply
\begin{equation}\label{eq:recgammaneg}
\frac{1}{\Gamma(-z)} = -\frac{ \sin{\pi z}}{\pi} \Gamma (z+1)
\end{equation}
It is plot is presented in Fig. \ref{fig:recgamaneg}.

\begin{figure}[hpt]
	\centering
	\includegraphics[width=0.7\linewidth]{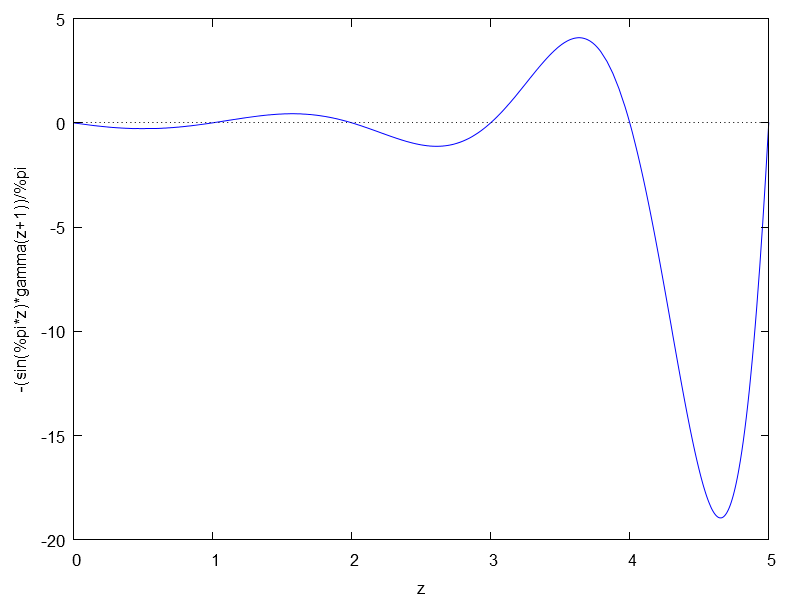}
	\caption{$1/\Gamma (-z)$ computed from Eq. \ref{eq:recgammaneg}}
	\label{fig:recgamaneg}
\end{figure}
\begin{figure}[hpt]
	\centering
	\includegraphics[width=0.5\linewidth]{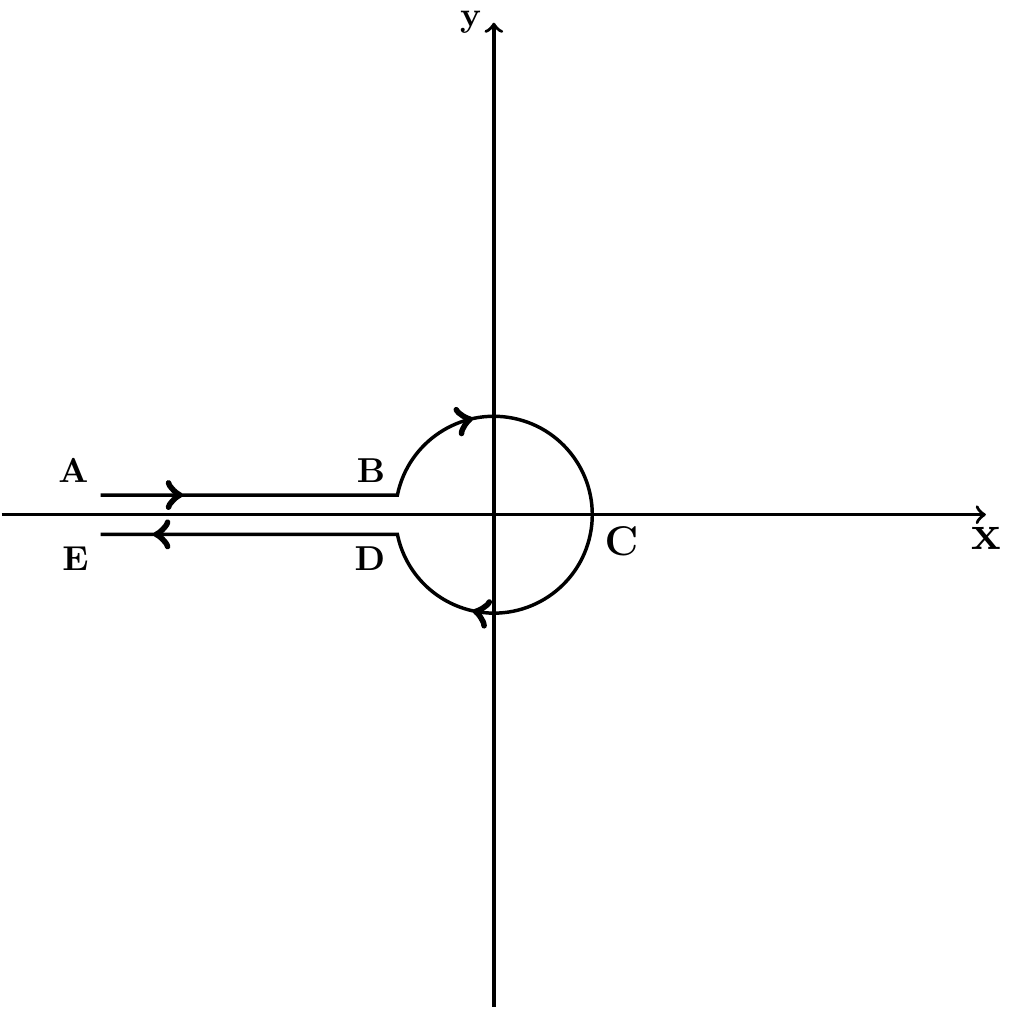}
	\caption{The Hankel contour  $Ha^{-}$}
	\label{fig:hankel}
\end{figure}

The Euler's Gamma function integral representation is valid for real $z>0$ or complex numbers, such that  $ Re \ z > 0$
\[
\Gamma (z)  =  \int_{0}^{\infty}  e^{-\tau}\tau^{z-1}\,  d\tau
\]
however for negative $z$ the integral diverges.
A less well-known integral representations for $Re z <0$ is  the Cauchy–Saalsch\"utz integral \cite{Temme1996}[Ch. 3]
\[
\Gamma (-z)=\int_{0}^{\infty} \frac{ e^{-\tau} - e_{n}(-\tau)}{\tau^{z+1}}  d \tau
\]

The Heine’s  complex representation for the reciprocal Gamma function is well known and is given below:
\[
\frac{1}{\Gamma (z)} =\frac{(-1)^{-z}}{2 \pi i}    \int_{Ha^{+}} \frac{e^{-\tau}}{\tau^z}\,  d\tau = \frac{ 1}{2 \pi i}  \int_{Ha^{-}} \frac{e^{\tau}}{\tau^z}\,  d\tau 
\]
Here ${  Ha^{-}}$ denotes the Hankel contour in the complex $\zeta$-plane with a cut along the negative real semi-axis $\arg \zeta =\pi$ and ${  Ha^{+}}$ is its reflection.
The contour is depicted in Fig. \ref{fig:hankel}.
The integrand has a simple pole at $\tau=0$.
The H\"older exponent at 0  can be computed in the closed interval $[0, \epsilon]$ as 
\[
  \llim{\epsilon}{0}{ }  \frac{\log e^{\epsilon} \epsilon^{-z}}{\log \epsilon } = -z + \llim{\epsilon}{0}{ }  \frac{\epsilon}{\log \epsilon}=-z
\]
Therefore, for $k> [z]$ it holds that 
$
 \llim{\epsilon}{0}{\epsilon^k e^{\epsilon} \epsilon^{-z} } = 0
$
and the order of the residue is $[z]$.
This observation is indicative for the statement of the main result of the paper. 

\subsection{Auxiliary notation}\label{sec:auxnot}
\begin{definition}
	For a real number $z$ the notation $[z]$ will mean the integral part of the number, while 
	$\{ z\}:= z- [z]  $ will denote the non-integral remaining part. 
\end{definition}

\begin{definition}\label{def:fallfact}
	The falling factorial is defined as
	\[
	(z)_n :=\prod_{k=0}^{n-1}  z-k 
	\]	
\end{definition}

\begin{definition}\label{def:truncexp}
	Let
	\[
	e_n (x):= \sum_{k=0}^{n} \frac{x^k}{k!} 
	\]
	be the truncated exponent polynomial sum under the convention $e_{-1}(x)=0$. 
\end{definition}

\section{Theoretical Results}\label{sec:results}

\begin{theorem}[Real Reciprocal Gamma representation]\label{th:gammarec}
	Let
	$z>0$, $z \notin \fclass{Z}{} $ and $n=[z]$. Then
	\[
	\frac{1}{\Gamma(z)}=   \frac{ \sin{\pi z}}{\pi} \int_{0}^{\infty} \frac{ e^{-x} - e_{n-1}(-x)}{x^{z}}  dx =
	Im \; \frac{ 1}{\pi} \int^{0}_{-\infty} \frac{ e^{x} - e_{n-1}(x)}{x^{z}}  dx 
	\]
	where the integrals are over the real axis.
\end{theorem}
\begin{proof}
	First, we establish two preliminary results.
	Consider the following limit of the form $0/0$ and apply  \textit{n} times l'H\^opital s'rule:
	\[
	L_z = \llim{x}{0}{} \frac{e^{x} - e_n(x)}{x^z} = \frac{1}{(z)_n}\llim{x}{0}{} \frac{e^{x} -1 }{  x^{z-n}}\]
	Another application of l'H\^opital s'rule leads to
	\[
	L_z =   \frac{1}{(z)_n \, (z-n)}\llim{x}{0}{}  e^{x}  x^{n+1 -z}
	\]
	Therefore,
	\[
	L_z = \left\lbrace 
	\begin{array}{ll}
	0, &   z< n+1\\
	\frac{1}{\Gamma(n+1)}, &  z=n +1 \\
	\infty, & z > n+1
	\end{array}
	\right.
	\] 
	Secondly, consider the limit
	\[
	M_z= \llim{x}{-\infty}{} \frac{e^{x} - e_n(x)}{x^z} =
	M_z= \llim{x}{-\infty}{} \frac{e^{x}}{x^z} - \sum_{k=0}^n \llim{x}{-\infty}{} \frac{ x^{k-z} }{\Gamma(k+1) } 
	\]
	Therefore,
	\[
	M_z = \left\lbrace 
	\begin{array}{ll}
	\infty, &   z< n\\
	\frac{1}{\Gamma(n+1)}, &  z=n  \\
	0, & z > n
	\end{array}
	\right.
	\] 
	Therefore, in order for both limits to vanish simultaneously $ n < z <n+1$. 
	Therefore, $n = [z]$. Let $\{z\} = z- [z]$.
	Then 
	\begin{multline*}
	I_{n+1} (z+1) =  \int^{0}_{-\infty} \frac{ e^{x} - e_n(x)}{x^{z+1}}  dx =
	\left. \frac{e^{x} - e_n(x) } {z x^{z}} \right|^{\infty}_{0} + \frac{1}{z} \int^{0}_{-\infty}  \frac{ \left( e^{x} - e_n(x)\right)^\prime }{x^{z}}  dx = \\ 
	\frac{1}{z} \int^{0}_{-\infty}  \frac{   e^{x} - e_{n-1}(x)  }{x^{z}}  dx = \frac{1}{z} I_{n} (z) 
	\end{multline*}
	by the above results.
	Therefore, by reduction
	\begin{multline*}
	I_{n+1} (z+1) = \frac{1}{(z)_{n}} I_0 (z- n) = \frac{1}{(z)_n} I_0 (\{z\})=
	\frac{1}{(z)_{n}} \int^{0}_{-\infty}   \frac{e^{x}}{x^{\{z\}}}  dx = \\
	\frac{1}{(z)_{n}} \int_{0}^{\infty}   \frac{e^{-x}}{(-x)^{\{z\}}} dx  =
	e^{-i \pi \{z\}  }\frac{\Gamma (1- \{z\})}{(z)_{n}}
	\end{multline*}
	Therefore,
	\begin{multline*}
	\Gamma ( \{z\}) I_{n+1} (z+1) = \frac{e^{-i \pi \{z\}  }}{(z)_{n}} \Gamma (1- \{z\}) \Gamma ( \{z\})  = \\ e^{-i \pi \{z\}  } \frac{\pi }{(z)_{n} \sin{\pi \{z\} }} = \frac{\pi }{(z)_{n}} (\cot{\pi \{z\}  }- i )
	\end{multline*}
	by the Euler's reflection formula.
	We take the imaginary part of the integral since $\Gamma(\{z\})$ is real and the middle expression is imaginary.
	Therefore,
	\[
	\frac{1}{\Gamma(z+1)}=\frac{1 }{ (z)_{n} \Gamma ( \{z\})} =  Im \frac{1}{\pi} I_{n+1} (z+1)
	\] 
	Finally,
	\[
	\frac{1}{\Gamma(z)}  = Im \frac{1}{\pi}   \int^{0}_{-\infty}  \frac{   e^{x} - e_{n-1}(x)  }{x^{z}} dx
	\]
\end{proof}

\begin{corollary}
	By change of variables it holds that 
	\[
	\frac{1}{\Gamma(z)}  =	\frac{\sin{\pi z}}{\pi z} \int_{0}^{\infty }{\left. {{u}^{\frac{1}{z}-2}}\, \left({  e}^{-{u}^{\frac{1}{z}}}-  \sum_{k=0}^{n-1}{\left. \frac{{{\left( -1\right) }^{k}}\, {{u}^{\frac{k}{z}}}}{k!}\right.} \right) \, du\right.}
	\]
\end{corollary}
The latter result can be useful for computations with large arguments of $\Gamma$.

\begin{corollary}[Modified Euler Integral of the second kind]
	By change of variables it holds that for $z>0$
	\[
	\frac{1}{\Gamma(z)}  = \frac{1}{\pi} \int_{0}^{1}{\left. \frac{u-    \sum\limits_{k=0}^{n-1}{\left. \frac{{{(\log{u})}^{k}}}{k!}\right.} }{u\, {\left( \log{u}\right) ^{z}}}du\right.} = \frac{\sin {\pi z}}{\pi}\int_{0}^{1}{\left. \frac{1-  e_{n-1} \left( \log u \right) /u     }{ {\left( \log{1/u}\right) }^{z}}du\right.}
	\]
\end{corollary}

Finally, it is instructive to demonstrate the correspondence between the complex-analytical representation and the hyper-singular representation.

\begin{theorem}[Regularized complex reciprocal Gamma representation]\label{th:comprep}
For		$z>0$, $z \notin \fclass{Z}{} $ and $n=[z]$
\[
\frac{1}{\Gamma(z)} = \frac{ 1}{2 \pi i}  \int_{Ha^{-}} \frac{e^{\tau} - e_{n-1} (\tau)}{\tau^z}\,  d\tau =  \frac{ \sin{\pi z}}{\pi} \int_{0}^{\infty} \frac{ e^{-\tau} - e_{n-1}(-\tau)}{\tau^{z}}  d \tau 
\]
where $\tau \in \fclass{R}{}$.
\end{theorem}
 \begin{proof}
	The proof technique follows \cite{Luchko2008}. 
	We evaluate the line integral along the Hankel contour: 
	\[
	I_n (z)	= \frac{ 1}{2 \pi i}  \int_{Ha^{-}} \frac{e^{\tau} - e_{n-1} (\tau)}{\tau^z}\,  d\tau 
	\]
	with kernel
	\[
	Ker(\tau)= \frac{e^{\tau} - e_{n-1} (\tau)}{\tau^z}, \quad z>0
	\]
	The contour is depicted in Fig. \ref{fig:hankel}.
	The integral can be split in three parts
	\[
	\int_{Ha} Ker(\tau) d \tau = \int_{AB} Ker(\tau) d \tau + \int_{BCD} Ker(\tau) d \tau + \int_{DE} Ker(\tau) d \tau
	\]
	along the rays AB, DE and the arch BCD, respectively. 
	Along the ray AB where $\tau=r e^{i \delta}$ the kernel becomes
	\[
	Ker_A =  \frac{1 }{{r}^{z}}  \left( {{  e}^{{{ e}^{  i \delta} -  i \delta z} r}}-\sum\limits_{k=0}^{n-1}{\left. \frac{{{  e}^{  i \delta k -  i \delta z}}\, {{r}^{k}}}{k!}\right.}\right)  
	\]
	Along the ray DE where $\xi=r e^{ - i \delta}$ the kernel becomes
	\[
	Ker_B =   \frac{1 }{{r}^{z}}  \left( {{  e}^{{{ e}^{ - i \delta} +  i \delta z} r}}-\sum\limits_{k=0}^{n-1}{\left. \frac{{{  e}^{ - i \delta k +  i \delta z}}\, {{r}^{k}}}{k!}\right.}\right)  
	\]
	Therefore,
	\begin{multline*}
	Ker_A - Ker_B = \\-\frac{2 \, i}{{r}^{z}} \left( {{ e}^{\cos{(\delta)} r}}\, \sin{\left( \delta z-\sin{(\delta)} r\right) }-\left( \sum\limits_{k=0}^{n-1}{\left. \frac{\cos{\left( \delta k\right) }\, {{r}^{k}}}{k!}\right.}\right) \, \sin{\left( \delta z\right) }+\left( \sum\limits_{k=0}^{n-1}{\left. \frac{\sin{\left( \delta k\right) }\, {{r}^{k}}}{k!}\right.}\right) \, \cos{\left( \delta z\right) }\right)
	\end{multline*}
	Therefore,
	\[
	\llim{\delta}{\pi}{}\frac{1}{2 \pi i} \int_{\infty}^{0} (	Ker_A-Ker_B )\, dr= 
	\frac{\sin (\pi z)}{ \pi }\int_{0}^{\infty} \frac{e^{-z} - e_{n-1} (-r)}{r^z}   dr
	\]

	The integral on the arc BCD is given by the Cauchy Residue Theorem.
 	By the above observation  the residue at $\tau=0$  is given by the limit
 	\[
	 Res[ Ker] (\tau) = \llim{\tau}{0}{ {{\tau}^{1-z}}\, \left({e}^{\tau}- e_{n-1} (\tau)\right)  } = \llim{\tau}{0}{\tau L_z}=0, \quad z \leq n+1 
	 \]
	 since $L_z = 0$.
	 Therefore,
	 \[
	 \int_{BCD} Ker(\tau) d \tau = 0
	 \]
	 Furthermore, after integration by parts
	 \[
	 I_n(z)= - \left. \frac{e^{\tau} - e_{n-1} (\tau)}{\tau^{z-1}} \right|_{Ha^{-}} + \frac{1}{2 \pi i  (z-1) }  \int_{Ha^{-}} \frac{e^{\tau} - e_{n-2} (\tau)}{\tau^{z-1}}\,  d\tau = \frac{1}{ z-1} I_{n-1}(z)
	 \]
	 since $M_z=0$.
	 Therefore, the claim follows by reduction to $n=0$.
\end{proof}

\subsection{Applications}\label{sec:appli}
As a concrete application of Th. \ref{th:comprep} consider the  Laplace transform pair
\begin{flalign*}
\mathcal{L}_s :  f(t) &\div  F (s) \\
				 t^k & \div  \frac{\Gamma(k+1)}{s^{k+1}}, \quad k > 0
\end{flalign*}
The inverse Laplace transform can be calculated simply as as
\[
\mathcal{L}_t^{-1}: \frac{\Gamma(k+1)}{s^{k+1}} \quad \div \quad \frac{1}{2 \pi i} \int_{Ha^{-}} \Gamma(k+1)\, \frac{e^{t s} - e_{[k]} (t  s) } {s^{k+1}} ds = t^k
\]
by change of variables.

On a second place,  the ratio of two Gamma functions can be represented as 
\begin{proposition}
Let $A, B>0$. Then
\[
\frac{\Gamma(A)}{\Gamma(B)}  =\frac{1}{\pi} \int_{0}^{1}\int_{0}^{1}  \frac{1- e_{n-1} \left( \log u \right) /u}{ {\left( \log{u}\right) }^{B}} { \left( - \log{v}\right) }^{A-1} du \, dv
\]	
where $n=[B]$.
\end{proposition}
\begin{proof}
\begin{multline*}
\frac{\Gamma(A+1)}{\Gamma(B)} =\frac{1}{\pi}  \int_{0}^{1}{\left. \frac{1-  e_{n-1} \left( \log u \right) /u     }{ {\left( \log{u}\right) }^{B}}du\right.} 
\int_{0}^{1}{ \left( - \log{v}\right) }^{A} dv= \\ 
\frac{1}{\pi} \int_{0}^{1}\int_{0}^{1}  \frac{1- e_{n-1} \left( \log u \right) /u}{ {\left( \log{u}\right) }^{B}} { \left( - \log{v}\right) }^{A} du \, dv
\end{multline*}	
\end{proof}

Finally, for negative arguments:
\begin{proposition}\label{th:neggamma}
	For $z>0$ it holds that
	\[
	\Gamma(-z) =- \frac{1}{z} \int_{0}^{\infty} \frac{ e^{-x} - e_{n-1}(-x)}{x^{z}}  dx
	\]
\end{proposition}
\begin{proof}
	By the reflection formula
	\[
	\Gamma(1-z) \Gamma (z) = \frac{\pi}{\sin{ \pi z}} = -z \Gamma (-z) \Gamma (z)
	\]
	Therefore,
	\[
	\Gamma(-z) = -\frac{\pi  }{z \sin{\pi z}} \frac{1}{\Gamma(z)}=- \frac{1}{z} \int_{0}^{\infty} \frac{ e^{-x} - e_{n-1}(-x)}{x^{z}}  dx
	\]
\end{proof}
\begin{remark}
	The latter result is equivalent to the classical Cauchy–Saalsch\"utz integral representation.
	Indeed, by integration by parts
	\begin{multline*}
	\Gamma(-z) = - \frac{1}{z} \int_{0}^{\infty} \frac{ e^{-x} - e_{n-1}(-x)}{x^{z}}  dx = \frac{1}{z} \int_{0}^{\infty} \frac{ d (e^{-x} - e_{n}(-x))}{x^{z}}    = \\ \left. \frac{   e^{-x} - e_{n}(-x)}{x^{z}} \right|^\infty_0 - \frac{- z}{z} \int_{0}^{\infty} \frac{ e^{-x} - e_{n}(-x)}{x^{z+1}}  dx = 
	\int_{0}^{\infty} \frac{ e^{-x} - e_{n}(-x)}{x^{z+1}}  dx
	\end{multline*}
	which is the Cauchy–Saalsch\"utz integral. 
\end{remark}

\section{Numerical Implementation}\label{sec:impl}

A reference implementation in the Computer Algebra System Maxima is given in Listings  \ref{lst:recgamma} and \ref{lst:neggamma}.
The numerical integration code uses the library Quadpack, distributed with Maxima. 
The reference implementation given in this section uses a routine for semi-infinite interval integration with tunable relative error of approximation (i.e. the \textit{epsrel} parameter).
A plot of the reciprocal $\Gamma$ function computed from Listing \ref{lst:recgamma} is presented in Fig. \ref{fig:recgamapos1}.
\begin{figure}[hpt]
	\centering
	\includegraphics[width=0.7\linewidth]{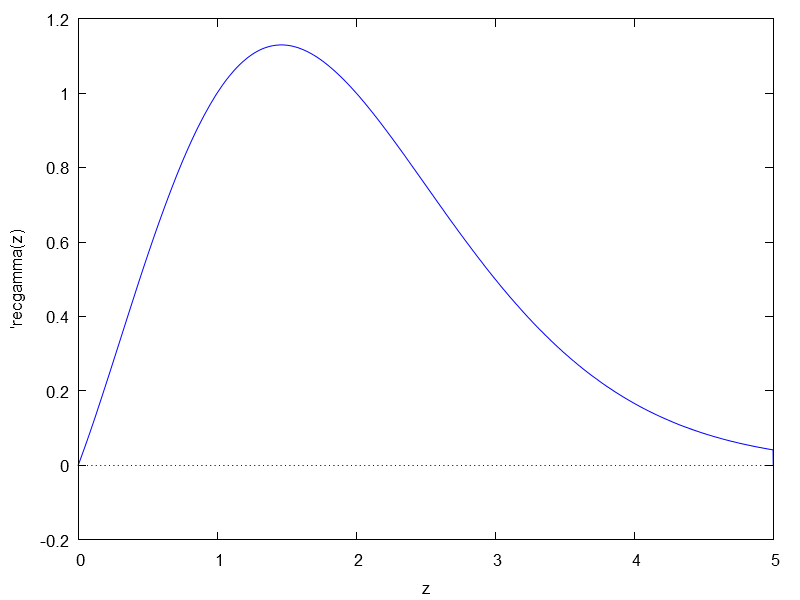}
	\caption{$1/ \Gamma(z)$ computed from Th. \ref{th:gammarec}}
	\label{fig:recgamapos1}
\end{figure}
\begin{figure}[hpt]
	\centering
	\includegraphics[width=0.7\linewidth]{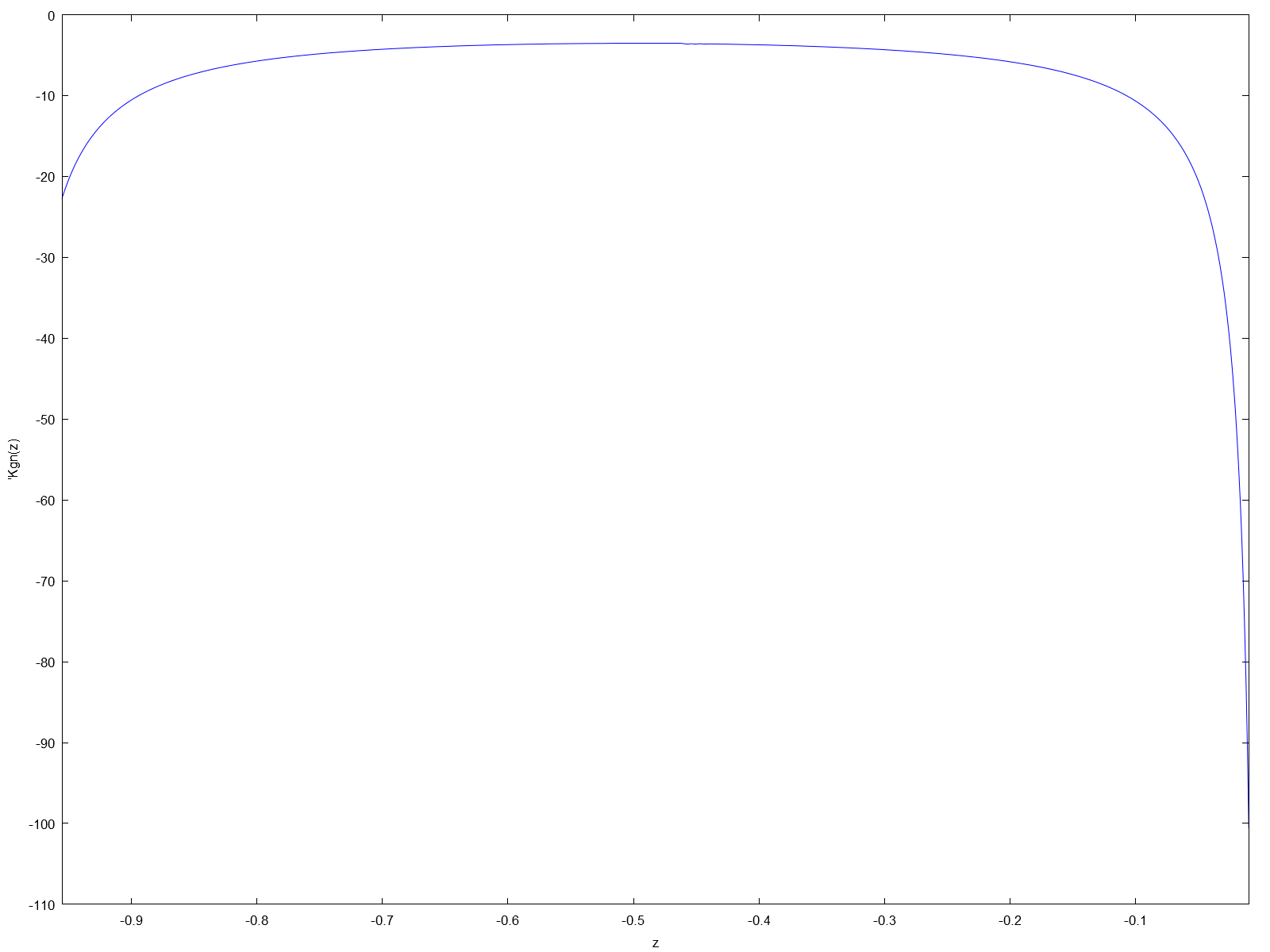}
	\caption{$\Gamma(-z)$ computed from Prop. \ref{th:neggamma}}
	\label{fig:recgamaneg1}
\end{figure}
Fig. \ref{fig:recgamaneg1} represents a plot of $\Gamma(-z)$ computed from Listing \ref{lst:recgamma}.

 \begin{lstlisting}[caption=The Maixma code corresponsing to Th. \ref{th:gammarec},
 label= lst:recgamma]
 
Kg(z):=block([ u, ret, k, n, fr:1, ss:0],
	if not numberp(z) then return('Kg(z)), 
	if integerp(z) then if z=0 then return(0) else return(1/(z-1)!)
	else (
		fr: sin(%pi*z)/%pi, 
		n:fix(z), 
		ss:sum( (-u)^k/k!,k,0, n-1),
		ret:fr*first(quad_qagi ( (exp(-u)-ss)/ u^(z) , u, 0, inf, 'epsrel=1d-8))
	),
	float(ret)
);
\end{lstlisting}

\begin{lstlisting}[caption=The Maixma code corresponsing to Corr. \ref{th:neggamma},
label= lst:neggamma]

Kgn(z):=block([ u, ret, k, n, fr:1, ss:0],
	if not numberp(z) then return('Kgn(z)), 
	if integerp(z) then   return(0)
	else (
		if z<0 then z:-z,
		fr:-1/z,
		n:fix(z), 
		ss:sum( (-u)^k/k!, k, 0, n-1),
		ret:fr*first(quad_qagi ( (exp(-u)-ss)/ u^(z) , u, 0, inf, 'epsrel=1d-8))
	),
	float(ret)
);
\end{lstlisting}

\section*{Acknowledgments}
The work has been supported in part by a grants from Research Fund - Flanders (FWO), contract number  VS.097.16N and the COST Association  Action CA16212 INDEPTH. 
Plots are prepared with the computer algebra system Maxima.

\bibliographystyle{plain}
\bibliography{biodiffusion1}

\begin{thebibliography}{1}

\bibitem{Borwein2018}
J.~M. Borwein and R.~M. Corless.
\newblock Gamma and factorial in the monthly.
\newblock {\em Am. Math. Monthly}, 125(5):400--424, apr 2018.

\bibitem{Davis1959}
P.~J. Davis.
\newblock {Leonhard Euler's integral: A historical profile of the Gamma
  function: In memoriam: Milton Abramowitz}.
\newblock {\em Am. Math. Monthly}, 66(10):849--869, 1959.

\bibitem{Luchko2008}
Y.~Luchko.
\newblock Algorithms for evaluation of the {Wright} function for the real
  arguments' values.
\newblock {\em Fract. Calc. Appl. Anal.}, 11(1):57 --75, 2008.

\bibitem{Mainardi2010a}
F.~Mainardi.
\newblock {\em Fractional Calculus and Waves in Linear Viscoelasticity}.
\newblock Imperial College Press, may 2010.

\bibitem{Prodanov2016a}
D.~Prodanov.
\newblock Regularization of derivatives on non-differentiable points.
\newblock {\em Journal of Physics: Conference Series}, 701(1):012031, 2016.

\bibitem{Temme1996}
Temme.
\newblock {\em Special Functions}.
\newblock John Wiley \& Sons, 1996.

\end{thebibliography}

\end{document}